\documentclass[11pt]{article}
\usepackage{amsfonts}
\usepackage{amsmath,amssymb}
\usepackage{amsthm}
\usepackage[mathscr]{euscript}
\usepackage{mathrsfs}
\usepackage{indentfirst}
\usepackage{color}
\usepackage{enumerate}

\newtheorem{theorem}{Theorem}[section]
\newtheorem{lemma}[theorem]{Lemma}

\numberwithin{equation}{section}

\newcommand{\lbl}[1]{\label{#1}}
\allowdisplaybreaks

\newcommand{\be}{\begin{equation}}
\newcommand{\ee}{\end{equation}}
\newcommand\bes{\begin{eqnarray}} \newcommand\ees{\end{eqnarray}}
\newcommand{\bess}{\begin{eqnarray*}}
\newcommand{\eess}{\end{eqnarray*}}
\newcommand{\bbbb}{\left\{\begin{aligned}}
\newcommand{\nnnn}{\end{aligned}\right.}
\newcommand{\bea}{\begin{align*}}
\newcommand{\eea}{\end{align*}}

\newcommand\ep{\varepsilon}

\newcommand\kk{\left}
\newcommand\rr{\right}
\newcommand\dd{\displaystyle}

\newcommand\dx{{\rm d}x}
\newcommand\dy{{\rm d}y}

\newcommand\yy{\infty}
\newcommand\qq{\eqref}

\newcommand\ccc{\color{blue}}

\begin{document}\thispagestyle{empty}
\setlength{\baselineskip}{16pt}

\begin{center}
 {\LARGE\bf Spreading dynamics of a Fisher-KPP nonlocal diffusion
model with a free boundary\footnote{This work was supported by NSFC Grants
12171120, 12301247}}\\[4mm]
 {\Large Lei Li}\\[0.5mm]
{College of Science, Henan University of Technology, Zhengzhou, 450001, China}\\[2.5mm]
{\Large Mingxin Wang\footnote{Corresponding author. {\sl E-mail}: mxwang@hpu.edu.cn}}\\[0.5mm]
 {School of Mathematics and Information Science, Henan Polytechnic University, Jiaozuo, 454000, China}
\end{center}

\date{\today}

\begin{abstract} This paper concerns the spreading speed and asymptotical behaviors, which was left as an open problem in \cite{LLW22}, of a Fisher-KPP nonlocal diffusion model with a free boundary. Using a new lower solution, we get the exact finite spreading speed of free boundary that is proved to be the asymptotical spreading speed of solution component $u$. Moreover, for the algebraic decay kernels, we derive the rates of accelerated spreading and the corresponding asymptotical behaviors of solution component $u$. Especially, for the level set $E^{\lambda}(t)$ with $\lambda\in(0,u^*)$, we find an interesting propagation phenomenon different from the double free boundary problem.

\textbf{Keywords}: Nonlocal diffusion; Free boundary; Spreading speed; Accelerated spreading.

\textbf{AMS Subject Classification (2020)}: 35K57, 35R09,
35R20, 35R35, 92D25
\end{abstract}

\section{Introduction}
The free boundary problems with local diffusion or random diffusion have been increasingly used to model propagation phenomenon in ecology since the pioneering work \cite{DL}. Lots of developments along different directions have been made over the past decade. Please see the expository article \cite{Du22} and the references therein for details.  Inspired by this, Cao et al \cite{CDLL} replaced the Laplacian operator in the model of \cite{DL} with an integral operator, and proposed the following free boundary problem with nonlocal diffusion
\begin{eqnarray}\left\{\begin{aligned}\label{1.0}
&u_t=d\int_{-\yy}^{\yy}\!\!J(x-y)\big[u(t,y)-u(t,x)\big]\dy+f(u), & &t>0,~g(t)<x<h(t),\\
&u(t,x)=0,& &t>0, ~ x\notin(g(t),h(t)),\\
&h'(t)=\mu\int_{g(t)}^{h(t)}\!\!\int_{h(t)}^{\infty}
J(x-y)u(t,x){\rm d}y{\rm d}x,& &t>0,\\
&g'(t)=-\mu\int_{g(t)}^{h(t)}\!\!\int_{-\infty}^{g(t)}
J(x-y)u(t,x){\rm d}y{\rm d}x,& &t>0,\\
&h(0)=-g(0)=h_0>0,\;\; u(0,x)=\tilde u_0(x),& &-h_0\le x\le h_0,
 \end{aligned}\right.
\end{eqnarray}
where $\tilde{u}_0\in C([-h_0,h_0])$, $\tilde{u}_0(\pm h_0)=0<\tilde{u}_0(x)$ in $(-h_0,h_0)$, and the kernel $J$ satisfies
 \begin{enumerate}
\item[({\bf J})]$J\in C(\mathbb{R})\cap L^{\infty}(\mathbb{R})$, $J\ge 0$, $J(0)>0,~\displaystyle\int_{\mathbb{R}}J(x){\rm d}x=1$, \ $J$\; is even.
 \end{enumerate}
The growth term $f$ is assumed to be of the Fisher-KPP type, i.e.,
  \begin{enumerate}
  \item[({\bf F})]$f\in C^1(\mathbb{R})$, $f(0)=0<f'(0)$, $f(u^*)=0>f'(u^*)$ for some $u^*>0$, and $\frac{f(u)}{u}$ is strictly decreasing in $u>0$.
 \end{enumerate}
 Meanwhile, Cor{\'t}azar et al. \cite{CQW} independently proposed problem \eqref{1.0} with $f(u)\equiv0$.  For some detailed explanations for this diffusion operator and the free boundary condition in \eqref{1.0}, one can refer to \cite{CDLL,CQW,AMRT}.  It was proved in \cite{CDLL} that the longtime behaviors of solution to \eqref{1.0} is governed by a spreading-vanishing dichotomy, which also holds for the free boundary problems with local diffusion. A striking difference comes from the spreading speed when spreading happens. More precisely, Du et al \cite{DLZ} proved that there exists a finite spreading speed for \eqref{1.0} if and only if a threshold condition is valid for kernel $J$. However, it can be learned from \cite{DL} or \cite{Du22} that the spreading speed of the free boundary problem with local diffusion is always finite.

 Motivated by  the above work, Li et al. \cite{LLW22} investigated the following free boundary problem
 \bes\label{1.1}
\left\{\begin{aligned}
&u_t=d\dd\int_{0}^{\yy}J(x-y)u(t,y)\dy-du+f(u), && t>0,~0\le x<h(t),\\
&u(t,x)=0,&& t>0,~ x\ge h(t),\\
&h'(t)=\mu\dd\int_{0}^{h(t)}\int_{h(t)}^{\infty}
J(x-y)u(t,x)\dy\dx,&& t>0,\\
&h(0)=h_0,\;\; u(0,x)=u_0(x),&& 0\le x\le h_0,
\end{aligned}\right.
 \ees
where $d,\mu$ and $h_0$ are positive constants, conditions {\bf (J)} and {\bf (F)} hold for kernel $J$ and nonlinear term $f$, respectively. The initial function $u_0$ is assumed to satisfy that $u_0\in C([0,h_0])$ and $u_0(h_0)=0<u_0(x)$ for $x\in[0,h_0)$. In model \qq{1.1}, the range $0\le x<h(t)$ of species indicates that species can only expand their habitat towards the right, while the left boundary $x=0$ is fixed and the species can cross this boundary but they will die once they do it, which is different from problem \eqref{1.0} where the species can extend their habitat in two directions. It has been shown in \cite{LLW22} that problem \eqref{1.1} has a unique global solution $(u,h)$ whose longtime behavior is governed by a spreading-vanishing dichotomy. Namely, one of the following alternatives happens for \eqref{1.1}:
\vspace{-1mm}
\begin{enumerate}
\item[{\rm(i)}]\, \underline{Spreading:} $\lim_{t\to\yy}h(t)=\yy$ and $\lim_{t\to\yy}u(t,x)=U(x)$ locally uniformly in $[0,\yy)$, where $U(x)$ is the unique bounded positive solution of
\bes\label{1.2}
  d\dd\int_{0}^{\yy}J(x-y)U(y)\dy-dU+f(U)=0 {\rm ~ ~ ~ in ~ }\;[0,\yy).
  \vspace{-2mm}\ees
\item[{\rm(ii)}]\, \underline{Vanishing:} $\lim_{t\to\yy}h(t)<\yy$ and $\lim_{t\to\yy}\|u(t,\cdot)\|_{C([0,h(t)])}=0$.
    \vspace{-2mm}
 \end{enumerate}

Some sharp criteria for spreading and vanishing were also derived in \cite{LLW22}. Moreover, when spreading happens, the authors proved that if kernel function $J$ satisfies
  \begin{enumerate}
\item[{\bf(J1)}]$\dd\int_{0}^{\yy}xJ(x)\dx<\yy$,
 \end{enumerate}
then the following rough estimate for the spreading speed of free boundary $h(t)$ holds:
\bess
\frac{U(0)}{u^*} c_\mu\le\liminf_{t\to\yy}\frac{h(t)}{t}\le\limsup_{t\to\yy}\frac{h(t)}{t}\le  c_\mu,
\eess
where positive constant $c_\mu$ is uniquely determined by the so-called semi-wave problem
\bes\label{1.3}\left\{\begin{array}{lll}
 d\dd\int_{-\yy}^{0}J(x-y)\phi(y)\dy-d\phi+c\phi'+f(\phi)=0, \quad -\yy<x<0,\\[1mm]
\phi(-\yy)=u^*,\ \ \phi(0)=0, \ \ c=\mu\dd\int_{-\yy}^{0}\int_{0}^{\yy}J(x-y)\phi(x)\dy\dx;
 \end{array}\right.
 \ees
 if $J$ violates {\bf (J1)}, then
 \bess
 \lim_{t\to\yy}\frac{h(t)}{t}=\yy,
 \eess
 which implies that the spreading speed is infinite and usually called the accelerated spreading.

In \cite{LLW22}, we conjectured that when spreading happens and {\bf (J1)} holds, $c_\mu$ is the exact spreading speed of $h(t)$, i.e.,
   \bess
 \lim_{t\to\yy}\frac{h(t)}{t}=c_\mu.
 \eess
 However, since solution component $u$ converges to the non-constant steady state solution $U(x)$ when spreading happens, in \cite{LLW22} we fail to construct some desired lower solutions to prove that $c_\mu$ is the exact spreading speed. In this paper, we introduce a new lower solution for \eqref{1.1} which help us to obtain the results as wanted. Besides, we get the asymptotical behaviors of solution component $u$ which show that $c_\mu$ is the asymptotical spreading speed of $u$. Below is our first main result.

\begin{theorem}\label{t1.1}Let $(u,h)$ be the unique solution of \eqref{1.1}. If  spreading happens, then we have
 \bess
 \left\{\begin{array}{lll}
 \dd\lim_{t\to\yy}\frac{h(t)}{t}=c_\mu ~ ~ {\rm and ~ ~ }\lim_{t\to\yy}\max_{x\in[0,ct]}|u(t,x)-U(x)|=0, ~ ~\forall\, c\in[0, c_\mu) ~ ~ {\rm if ~ {\bf (J1)} ~ holds},\\[3mm]
  \dd\lim_{t\to\yy}\frac{h(t)}{t}=\yy ~ ~ {\rm and ~ ~ }\lim_{t\to\yy}\max_{x\in[0,ct]}|u(t,x)-U(x)|=0, ~ ~\forall\, c\ge0 ~ ~ {\rm if ~ {\bf (J1)} ~ is ~ violated}.
  \end{array}\right.
 \eess
 \end{theorem}

We also check the longtime behaviors of the level set
   $$E_{\lambda}(t)=\{x\ge0: u(t,x)=\lambda, ~ \forall\,\lambda\in(0,u^*)\},$$
and find an interesting propagation phenomenon. That is, for $\lambda\in[U(0),u^*)$, $\inf E_{\lambda}(t)$ will converges to a unique nonnegative constant, while for $\lambda\in (0,U(0))$, $\inf E_{\lambda}(t)/t$ will converges to $c_\mu$ or infinity, depending on whether {\bf (J1)} holds for kernel $J$.

 \begin{theorem}\label{t1.2}Assume that spreading happens. Let $(u,h)$ be the unique solution of \eqref{1.1}. Then the following statements hold.
\begin{enumerate}[$(1)$]
\item For any $\lambda\in[U(0),u^*)$, we have
\[\lim_{t\to\yy}\inf E_{\lambda}(t)=X_{\lambda},\]
where $X_{\lambda}$ is uniquely determined by $U(X_{\lambda})=\lambda$. This clearly implies
\[\lim_{t\to\yy}\frac{\inf E_{\lambda}(t)}{t}=0.\]
\item For any $\lambda\in(0,U(0))$, we have
\bess
 \left\{\begin{array}{lll}
 \dd\lim_{t\to\yy}\frac{\inf E_{\lambda}(t)}{t}=c_\mu\; ~ ~ {\rm if ~ {\bf (J1)} ~ holds},\\[3mm]
  \dd\lim_{t\to\yy}\frac{\inf E_{\lambda}(t)}{t}=\yy\; ~ ~ {\rm if ~ {\bf (J1)} ~ is ~ violated}.
  \end{array}\right.
 \eess
\item For any $\lambda\in(0,u^*)$, we have
\bess
 \left\{\begin{array}{lll}
 \dd\lim_{t\to\yy}\frac{\sup E_{\lambda}(t)}{t}=c_\mu ~ ~ {\rm if ~ {\bf (J1)} ~ holds},\\[3mm]
  \dd\lim_{t\to\yy}\frac{\sup E_{\lambda}(t)}{t}=\yy ~ ~ {\rm if ~ {\bf (J1)} ~ is ~ violated}.
  \end{array}\right.
 \eess
\end{enumerate}
 \end{theorem}

 Then we study the rate of accelerated spreading for a class of algebraic decay kernel functions, namely,
 $J$ satisfies
\begin{enumerate}
\item[${\bf(J^\gamma)}$] there exist positive constants $\varsigma_1$ and $\varsigma_2$ such that $\varsigma_1|x|^{-\gamma}\le J(x)\le \varsigma_2|x|^{-\gamma}$ for $|x|\gg 1$.
 \end{enumerate}
 Due to condition {\bf (J)}, we have $\gamma>1$. Clearly, {\bf (J1)} holds if and only if $\gamma>2$. We focus on the case $\gamma\in(1,2]$, which, by Theorem \ref{t1.1}, implies that accelerated spreading happens. In order to simplify our presentation, we make some notations. For any given two functions $s(t)$ and $r(t)$, we say $s(t)\approx r(t)$ if there exist two positive constants $c_1$ and $c_2$ such that $c_1r(t)\le s(t)\le c_2r(t)$ for $t\gg 1$; we say $s(t)=o(r(t))$ if $\lim_{t\to\infty}\frac{s(t)}{\gamma(t)}=0$. Here is our main related result.

 \begin{theorem}\label{t1.3}Suppose that spreading happens and ${\bf(J^\gamma)}$ holds with $\gamma\in(1,2]$. Let $(u,h)$ be the unique solution of \eqref{1.1}. Then we have
 \bess
 \left\{\begin{array}{lll}
 \dd h(t)\approx t^{\frac{1}{\gamma-1}}, ~ ~\; \lim_{t\to\yy}\max_{x\in[0,s(t)]}|u(t,x)-U(x)|=0 {\rm ~ for ~ any ~ }s(t)=o(t^{\frac{1}{\gamma-1}}) \;~ {\rm if ~}\gamma\in(1,2),\\
  \dd h(t)\approx t\ln t, ~  ~\; \lim_{t\to\yy}\max_{x\in[0,s(t)]}|u(t,x)-U(x)|=0 {\rm ~ for ~ any ~ }s(t)=o(t\ln t)\; ~ {\rm if ~}\gamma=2.
  \end{array}\right.
 \eess
 \end{theorem}

 As an application of our above result, we also investigate the asymptotical behaviors of the following nonlocal diffusion problem on half line
 \bes\left\{\begin{aligned}\label{1.6}
&w_t= d\int_{0}^{\yy}J(x-y)w(t,y)\dy-dw+f(w), && t>0,\ x\ge0,\\
&w(0,x)=w_0(x),  && x\ge0,
 \end{aligned}\right.\ees
 where $w_0\in C_{c}([0,\yy))$ and $w_0(x)\ge0,\not\equiv0$ in $[0,\yy)$. It is not hard to show that problem \eqref{1.6} can be seen as the limiting problem of \eqref{1.1} as $\mu\to\yy$ and the solution component $u$ of \eqref{1.1} is a lower solution to this problem. Thus we can use Theorem \ref{t1.1} to derive the asymptotical behaviors of \eqref{1.6}. Please see Theorem \ref{t2.1} for the details.

Before finishing the introduction, we would like to mention the recent progress made on nonlocal diffusion problem \eqref{1.0} and its variations. Du and Ni \cite{DN1,DN2} studied the rate of accelerated spreading for \eqref{1.0} by constructing some subtle upper and lower solutions. In \cite{DN3}, they also discussed the approximation of the model in \cite{DL} by the variation of \eqref{1.0}. The time-periodic version of \eqref{1.0} was considered by \cite{ZLZ}. Moreover, for the system, one can refer to \cite{DN4,DN5,ZZLD,NV,PLL,LLW24,LW24} for the cooperative system and \cite{LSW,DWZ,DNS1,DNS2} for the competition or the prey-predator system.

This paper is arranged as follows. Section 2 involves the proofs of Theorem \ref{t1.1} and \ref{t1.2}. Section 3 is devoted to the proof of Theorem \ref{t1.3}. Throughout this paper, we always assume that {\bf (J)} and {\bf (F)} hold for kernel $J$ and nonlinear term $f$, respectively.

\section{Spreading speed}
\setcounter{equation}{0} {\setlength\arraycolsep{2pt}
In this section, we first prove Theorem \ref{t1.1} by resorting to a new lower solution, and then use the results in Theorem \ref{t1.1} and some basic analysis to show Theorem \ref{t1.2}. The following comparison principle will be often used in this paper.

\begin{lemma}\label{l2.1} For any $T>0$, we assume that $s(t)$ and $r(t)$ are continuous and non-decreasing in $[0,T]$, as well as $\underline{h}(t),\; \bar{h}(t)\in C^1([0,T])$ with $s(t)\le r(t)\le \min\{\underline{h}(t),\bar{h}(t)\}$ for $t\in[0,T]$. Let $\underline{u}\in C([0,T]\times[s(t),\underline{h}(t)])$, $\bar{u}\in C([0,T]\times[s(t),\bar{h}(t)])$, the left derivatives $\underline{u}_t(t-0,x)$ and $\bar{u}_t(t-0,x)$ exist in $(0,T]\times[r(t),\underline{h}(t)]$ and $(0,T]\times[r(t),\bar{h}(t)]$, respectively. Suppose that
 \bes
 \underline{u}(t,x)\le \bar u(t,x)\;\;\;{\rm for}\;\; 0<t\le T, ~ x\in[s(t),r(t)],
 \lbl{2b.1}\ees
and $(\underline{u},\underline{h})$ and $(\bar{u},\bar{h})$ satisfy
\bess
\left\{\begin{aligned}
&\underline u_t(t-0,x)\le d\displaystyle\int_{s(t)}^{\underline h(t)}J(x-y)\underline u(t,y){\rm d}y-d\underline u+ f(\underline u), && 0<t\le T,~x\in[r(t),\underline h(t)),\\
&\underline u(t,\underline h(t))\le0,&& 0<t\le T,\\
&\underline h'(t)<\mu\displaystyle\int_{s(t)}^{\underline h(t)}\int_{\underline h(t)}^{\infty}
J(x-y)\underline u(t,x){\rm d}y{\rm d}x,&& 0<t\le T,\\
&\underline h(0)<\bar h(0),\;\;\underline u(0,x)\le \bar u(0,x),&& x\in[s(0),\underline h(0)]
\end{aligned}\right.
\eess
and
\bess
\left\{\begin{aligned}
&\bar u_t(t-0,x)\ge d\displaystyle\int_{s(t)}^{\bar h(t)}J(x-y)\bar u(t,y){\rm d}y-d\bar u+ f(\bar u), && 0<t\le T,~x\in[r(t),\bar h(t)),\\
&\bar u(t,\bar h(t))\ge0,&& 0<t\le T,\\
&\bar h'(t)\ge\mu\displaystyle\int_{s(t)}^{\bar h(t)}\int_{\bar h(t)}^{\infty}
J(x-y)\bar u(t,x){\rm d}y{\rm d}x,&& 0<t\le T,
\end{aligned}\right.
\eess
respectively. Then we have
\[\underline{h}(t)\le\bar{h}(t), ~ ~ ~ \underline{u}(t,x)\le \bar{u}(t,x)\;\; ~ {\rm for ~ }\;(t,x)\in[0,T]\times[s(t),\underline{h}(t)].\]
\end{lemma}
\begin{proof}We start with proving $\underline{h}(t)\le \bar{h}(t)$ for $t\in[0,T]$. Otherwise, since $\underline{h}(0)<\bar{h}(0)$, there exists a $t^*\in(0,T]$ such that $\underline{h}(t)<\bar{h}(t)$ for $t\in[0,t^*)$ and $\underline{h}(t^*)=\bar{h}(t^*)$. Clearly, $\underline{h}'(t^*)\ge\bar{h}'(t^*)$.

Now we compare $\underline{u}$ with $\bar{u}$ for $t\in[0,t^*]$ and $x\in[s(t),\underline{h}(t)]$. Denote $\tilde{u}=\bar{u}-\underline{u}$. By the inequalities satisfied by $\underline{u}$ and $\bar{u}$, we have
 \bess
\left\{\begin{aligned}
&\tilde u_t(t-0,x)\ge d\displaystyle\int_{r(t)}^{\underline h(t)}J(x-y)\tilde u(t,y){\rm d}y+c(t,x)\tilde u, && 0<t\le t^*,~x\in[r(t),\underline h(t)),\\
&\tilde u(t,\underline h(t))\ge0,~ \tilde u(t,r(t))\ge0&& 0<t\le t^*,\\
&\tilde{u}(0,x)\ge0, &&x\in[r(0),\underline{h}(0)],
\end{aligned}\right.
\eess
where $c(t,x)\in L^{\yy}$.  Let $\tilde{U}=\tilde u {\rm e}^{-kt}$ with $k>\|c\|_{\yy}+d$. If $\min_{[0,t^*]\times[r(t),\underline{h}(t)]}\tilde{U}(t,x)<0$, then there exists $(t_0,x_0)\in(0,t^*]\times(r(t),\underline{h}(t))$ such that $\tilde{U}(t_0,x_0)=\min_{[0,t^*]\times[r(t),\underline{h}(t)]}\tilde{U}(t,x)<0$. Certainly, $\tilde{U}_t(t_0-0,x_0)\le0$ and
 \[\int_{r(t_0)}^{\underline h(t_0)}J(x_0-y)\tilde U(t_0,y){\rm d}y-\int_{r(t_0)}^{\underline h(t_0)}J(x_0-y)\dy\tilde U(t_0,x_0)\ge0.\]
It follows that
 \bess
0&\ge&\tilde{U}_t(t_0-0,x_0)-d\int_{r(t_0)}^{\underline h(t_0)}J(x_0-y)\tilde U(t_0,y){\rm d}y+d\int_{r(t_0)}^{\underline h(t_0)}J(x_0-y)\dy\tilde U(t_0,x_0)\\
&\ge& d\int_{r(t_0)}^{\underline h(t_0)}J(x_0-y)\dy\tilde U(t_0,x_0)+(c(t_0,x_0)-k)\tilde{U}(t_0,x_0)\\
&>&0.
 \eess
This contradiction indicates that $\tilde{u}(t,x)\ge0$ for $t\in[0,t^*]$ and $x\in[r(t),\underline{h}(t)]$. This combined with \qq{2b.1} allows us to derive  $\tilde{u}(t,x)\ge0$ for $t\in[0,t^*]$ and $x\in[s(t),\underline{h}(t)]$. According to the differential inequalities of $\underline{h}'$ and $\bar{h}'$ we obtain
\bess
0\ge\bar{h}'(t^*)-\underline{h}'(t^*)>\mu\int_{s(t^*)}^{\underline{h}(t^*)}\int_{\underline{h}(t^*)}^{\yy}J(x-y)\tilde{u}(t^*,x)\dy\dx\ge0.
\eess
This contradiction implies that $\underline{h}(t)\le \bar{h}(t)$ for $t\in[0,T]$. Then similar to the above, we can deduce that $\underline{u}(t,x)\le \bar{u}(t,x)$ by comparing $\underline{u}$ and $\bar{u}$ over $[0,T]\times[s(t),\underline{h}(t)]$. The proof is ended.
\end{proof}

For the steady state problem \eqref{1.2}, we have the following result.

\begin{lemma}\label{l2.2}
Problem \eqref{1.2} has a unique bounded positive solution $U\in C([0,\yy))$, and $U(x)\to u^*$ as $x\to\yy$. Moreover, $0<U<u^*$ and $U(x)$ is strictly increasing in $[0,\yy)$.
\end{lemma}

\begin{proof}The existence and uniqueness directly follow from \cite[Lemma 2.7]{LLW22}, and the monotonicity can be derived by using the similar methods as in the proof of \cite[Theorem 2.7]{LLW2}. The details are omitted here.
\end{proof}

For later use, we now recall some results which can be seen from the proof of \cite[Theorem 3.4]{LLW22} and will be used to construct the suitable lower solution. Let
\begin{align*}
\xi(x)=\left\{\begin{aligned}
&1,& &|x|\le1,\\
&2-|x|,& &1\le|x|\le2,\\
&0, & &|x|\ge2,
\end{aligned}\right.
  \end{align*}
and define $J_n(x)=\xi(\frac{x}{n})J(x)$. Clearly, $J_n$ are supported compactly for all $n\ge1$ and non-decreasing in $n$, $J_n\le J$, and $J_n\to J$ in $L^1(\mathbb{R})\cap C_{\rm loc}(\mathbb{R})$ as $n\to\yy$.
Thus we can choose $n$ to be large enough, say $n\ge N>0$, such that $d(\|J_n\|_1-1)u+f(u)$ still satisfies the condition {\bf(F)}.

For any $n\ge N$, the semi-wave problem
\bes\label{2.1}\left\{\begin{array}{lll}
 d\dd\int_{-\yy}^{0}\!J_n(x-y)\phi_n(y)\dy-d\phi_n+c_n\phi_n'+ f(\phi_n)=0,\;\;-\yy<x<0,\\[2mm]
\phi_n(-\yy)=u^*_{n},\ \ \phi_n(0)=0, \ \ c_n=\mu\dd\int_{-\yy}^{0}\int_{0}^{\yy}\!J_n(x-y)\phi_n(x)\dy\dx
 \end{array}\right.
 \ees
has a unique solution pair $( c_{n,\mu},\phi_n)$ with $ c_{n,\mu}>0$ and $\phi_n$ strictly decreasing in $(-\yy,0]$, where $u^*_{n}$ is the unique positive root of equation
 \[d(\|J_n\|_1-1)u+f(u)=0.\]
Clearly, $u^*_{n}\le u^*_{n+1}\le u^*$ and $\lim_{n\to\yy}u_n^*=u^*$. Moreover, if {\bf (J1)} holds, then $\lim_{n\to\yy} c_{n,\mu}= c_\mu$ and $\lim_{n\to\yy}\phi_n=\phi^{c_\mu}$ locally uniformly in $(-\yy,0]$, where $(c_{\mu},\phi^{c_\mu})$ is the unique solution pair of semi-wave problem \eqref{1.3}.

On the other hand, the steady state problem
 \bes\label{2.2}
  d\dd\int_{0}^{\yy}J_n(x-y)U(y)\dy-dU+f(U)=0 {\rm ~ ~ ~ in ~ }\;[0,\yy)
  \ees
has a unique bounded positive solution $U_n\in C([0,\yy))$, which is strictly increasing in $[0,\yy)$, $0<U_n<u^*_n$, $\lim_{x\to\yy}U_n(x)=u^*_n$, and $U_n\to U$ in $C_{\rm loc}([0,\yy))$ as $n\to\yy$.  The next result shows that $U_n\to U$ in $L^{\yy}([0,\yy))$ as $n\to\yy$ which will be used later.

\begin{lemma}\label{l2.3}Let $U$ and $U_n$ be the unique bounded positive solutions of \eqref{1.2} and \eqref{2.2}, respectively. Then $U_n\to U$ in $L^{\yy}([0,\yy))$ as $n\to\yy$.
\end{lemma}
\begin{proof}
Note that $U_n$ is strictly increasing in $n\ge1$. It thus remains to show that for any small $\ep>0$, there exists a large $N>0$ such that
 \bess
  U_n(x)\ge U(x)-\ep ~ ~ {\rm for ~ }n\ge N, ~ x\ge0.\eess
In fact, due to $u^*_n\to u^*$ as $n\to\yy$, one can find a $N>1$ such that $u^*_n>u^*-\ep/2$ for $n\ge N$. On the other hand, since $U_N(x)\to u^*_N>u^*-\ep/2$ as $x\to\yy$, there is a $X\gg 1$ such that $U_N(x)\ge u^*-\ep$ for $x\ge X$. Since $U_n$ is nondecreasing in $n$, $U_n(x)\ge u^*-\ep> U(x)-\ep$ for $x\ge X$ and $n\ge N$.

Moreover, owing to the fact that $U_n\to U(x)$ in $C_{\rm loc}([0,\yy))$ as $n\to\yy$, we can choose $n$ to be sufficiently large, say $n\ge N_1>N$, satisfying $U_n(x)\ge U(x)-\ep$ for $n\ge N_1$ and $x\in[0,X]$. Thus we have $U_n(x)\ge U(x)-\ep$ for $x\ge0$ if $n\ge N_1$.
\end{proof}

Now we are going to construct a suitable lower solution, which is different from the one in \cite{LLW22}, to show the desired result.

\begin{lemma}\label{l2.4}Let $(u,h)$ be the unique solution of \eqref{1.1}. Suppose that spreading happens and {\bf (J1)} holds. Then
\bess
\lim_{t\to\yy}\frac{h(t)}{t}=c_\mu,
\eess
where $c_\mu$ is uniquely given by semi-wave problem \eqref{1.3}.
\end{lemma}
\begin{proof} By \cite[Theorem 3.4]{LLW22}, we know $\limsup_{t\to\yy}{h(t)}/{t}\le c_\mu$.
It thus remains to prove
\bess
\liminf_{t\to\yy}\frac{h(t)}{t}\ge c_\mu.
\eess

We will prove it by a comparison argument. For any given $n\ge N$, by the definition of $J_n$, we know ${\rm supp}J_n=[-2n,2n]$. For any small $\ep>0$, there exists a $X_{\ep}>0$ such that
\[U_n(x)>(1-{\ep}/{2})u^*_n, ~ ~ \forall\, x\ge X_{\ep}.\]
Since spreading happens for \eqref{1.1}, we have $u(t,x)\to U(x)\ge U_n(x)$ in $C_{\rm loc}([0,\yy))$ as $t\to\yy$. So there exists a $T>0$ such that
\bess
\left\{\begin{aligned}
&h(T)>X_{\ep}+4n,\\
&u(t,x)\ge U(x)-\frac{\ep u^*_n}{2}\ge U_n(x)-\frac{\ep u^*_n}{2}\ge(1-\ep)u^*_n, ~ ~ \forall\, t\ge T, ~ x\in[X_{\ep},X_{\ep}+4n].
\end{aligned}\right.
 \eess
Define
\[\underline{h}(t)=(1-2\ep)c_{n,\mu}(t-T)+X_{\ep}+4n ~ ~ {\rm and } ~ ~ \underline{u}(t,x)=(1-\ep)\phi_n(x-\underline{h}(t)),\]
where $( c_{n,\mu},\phi_n)$ is the unique solution pair of \eqref{2.1}.

Next we are going to show that $(\underline{u},\underline{h})$ satisfies
\bes\label{2.4}
\left\{\begin{aligned}
&\underline u_t\le d\dd\int_{X_{\ep}}^{\underline h(t)}J_n(x-y)\underline u(t,y)\dy-d\underline u+ f(\underline u), && t>T,~X_{\ep}
+4n\le x<\underline h(t),\\
&\underline u(t,x)\le u(t,x),~ \underline{u}(t,\underline{h}(t))\le0, && t\ge T, ~x\in [X_{\ep},X_{\ep}+4n],\\
&\underline h'(t)<\mu\dd\int_{X_{\ep}}^{\underline h(t)}\int_{\underline h(t)}^{\infty}
J_n(x-y)\underline u(t,x)\dy\dx,&& t>T,\\
&\underline h(T)< h(T).
\end{aligned}\right.
 \ees
On the other hand, in view of \eqref{1.1}, we have
 \bess
\left\{\begin{aligned}
&u_t>d\dd\int_{X_{\ep}}^{h(t)}J_n(x-y)u(t,y)\dy-du+f(u), && t>T,~X_{\ep}+4n\le x<h(t),\\
&u(t,x)\ge\underline u(t,x) ,~  u(t,h(t))=0, &&t>T,~ x\in[X_{\ep},X_{\ep}+4n],\\
&h'(t)\ge\mu\dd\int_{X_{\ep}}^{h(t)}\int_{h(t)}^{\infty}
J_n(x-y)u(t,x)\dy\dx,&& t>T,\\
&h(T)>\underline h(T).
\end{aligned}\right.
 \eess
 Hence once \eqref{2.4} is proved, it follows from Lemma \ref{l2.1} with $(s(t),r(t))=(X_{\ep},X_{\ep}+4n)$ that $\underline{h}(t)\le h(t)$ and $\underline{u}(t,x)\le u(t,x)$ for $t\ge0$ and $x\in[X_{\ep},\underline{h}(t)]$.
 Together with the definition of $\underline{h}$, it yields that
 \bess
 \liminf_{t\to\yy}\frac{h(t)}{t}\ge\lim_{t\to\yy}
 \frac{\underline{h}(t)}{t}=(1-\ep)c_{n,\mu}.
 \eess
 Due to the arbitrariness of $\ep$ and the fact that $c_{n,\mu}\to  c_\mu$ as $n\to\yy$, we obtain the result as wanted.

Now we show that \eqref{2.4} is valid. Recall that ${\rm supp}J_n=[-2n,2n]$ and $\phi'_n\le 0$. Straightforward computations give that for $(t,x)\in(T,\yy)\times[X_{\ep}+4n,\underline{h}(t))$,
 \bess
 \underline{u}_t&=&-(1-\ep)(1-2\ep)\phi'_n(x-\underline{h}(t)) c_{n,\mu}\\
 &\le&-(1-\ep)\phi'_n(x-\underline{h}(t))c_{n,\mu}\\
 &=&(1-\ep)\bigg(d\int_{-\yy}^{0}J_n(x-\underline{h}-y)\phi_n(y)\dy-d\phi_n+f(\phi_n)\bigg)\\
 &=&(1-\ep)\bigg(d\int_{X_{\ep}}^{\underline{h}}J_n(x-y)\phi_n(y-\underline{h})\dy-d\phi_n+f(\phi_n)\bigg)\\
 &\le&d\int_{X_{\ep}}^{\underline{h}}J_n(x-y)\underline{u}(t,y)\dy-d\underline{u}+f(\underline{u}).
 \eess
The inequality in the first line of \eqref{2.4} holds. Clearly, $\underline{u}(t,\underline{h}(t))=(1-\ep)\phi_n(0)=0$. By the choice of $T$, we see
 \[u(t,x)\ge U_n(x)-\frac{\ep u^*_n}{2}\ge u^*_n(1-\ep)> (1-\ep)\phi_n(x-\underline{h})=\underline{u}(t,x),~ ~ \forall\, t\ge T, ~ x\in[X_{\ep},X_{\ep}+4n].\]
The inequalities in the second line of \eqref{2.4} are valid.
 Moreover,
 \bess
\mu\dd\int_{X_{\ep}}^{\underline h(t)}\int_{\underline h(t)}^{\infty}
J_n(x-y)\underline u(t,x)\dy\dx
&=&\mu(1-\ep)\dd\int_{X_{\ep}}^{\underline h(t)}\int_{\underline h(t)}^{\infty}
J_n(x-y)\phi_n(x-\underline{h}(t))\dy\dx\\
&=&\mu(1-\ep)\dd\int_{X_{\ep}-\underline{h}(t)}^{0}\int_{0}^{\infty}
J_n(x-y)\phi_n(x)\dy\dx\\
&=&\mu(1-\ep)\dd\int_{-(1-2\ep)c_{n,\mu}(t-T)-4n}^{0}\int_{0}^{\infty}
J_n(x-y)\phi_n(x)\dy\dx\\[1mm]
&=&(1-\ep) c_{n,\mu}>(1-2\ep) c_{n,\mu}\\
&=&\underline{h}'(t).
 \eess
The inequality in the third line of \eqref{2.4} is proved. It is obvious that $\underline h(T)< h(T)$. Hence, \eqref{2.4} holds and the proof is finished.
\end{proof}

Then we prove the asymptotical behaviors of solution component $u$.

\begin{lemma}\label{l2.5} Let $(u,h)$ be the unique solution of \eqref{1.1}. If  spreading happens, then
\bess
 \left\{\begin{array}{lll}
 \dd\lim_{t\to\yy}\max_{x\in[0,ct]}|u(t,x)-U(x)|=0, ~ ~\forall\, c\in[0, c_\mu) ~ ~ {\rm if ~ {\bf (J1)} ~ holds},\\
  \dd\lim_{t\to\yy}\max_{x\in[0,ct]}|u(t,x)-U(x)|=0, ~ ~\forall\, c\ge0 ~ ~ {\rm if ~ {\bf (J1)} ~ is ~ violated}.
  \end{array}\right.
 \eess
\end{lemma}
\begin{proof}This proof will be divided into three steps.

{\bf Step 1.} In this step, we show
\bess
\limsup_{t\to\yy}u(t,x)\le U(x) ~ ~ {\rm uniformly ~ in ~ }[0,\yy).
\eess
Obviously, it is sufficient to prove that for any small $\ep>0$, there exists a $T>0$ such that $u(t,x)\le U(x)+\ep$ for $t\ge T$ and $x\ge0$.
Recall that $U(x)\nearrow u^*$ as $x\to\yy$. For any small $\ep>0$, there is a $X\gg 1$ such that
  \bes\label{2.6}
   U(x)+\ep\ge u^*+\frac{\ep}{2} ~ ~ {\rm for ~ }x\ge X.
   \ees
Consider the initial value problem of ODE
\[\hat u'=f(\hat{u}),\;\;t>0;\;\; ~ \hat{u}(0)=\|u_0\|_{C([0,h_0])}.\]
Clearly, $\hat{u}(t)\to u^*$ as $t\to\yy$. A comparison argument gives $u(t,x)\le \hat{u}(t)$ for $t\ge0$ and $x\ge0$, which implies $
\limsup_{t\to\yy}u(t,x)\le u^*$ uniformly in $[0,\yy)$.
Hence, there exists a $T_1>0$ such that
\[u(t,x)\le u^*+{\ep}/2 ~ ~ {\rm for ~ }t\ge T_1, ~ x\ge0.\]
Combining with \eqref{2.6}, we get
\[u(t,x)\le u^*+{\ep}/2\le U(x)+\ep ~ ~ {\rm for ~ }t\ge T_1, ~ x\ge X.\]
Recall that $u(t,x)\to U$ in $C_{\rm loc}([0,\yy))$. One can find a $T_2>T_1$ such that $u(t,x)\le  U(x)+\ep$ for $t\ge T_2$ and $x\in[0,X]$.
So $u(t,x)\le U(x)+\ep$ for $t\ge T_2$ and $x\in[0,\yy)$.
This step is ended.

{\bf Step 2.} In this step, we show that if {\bf (J1)} holds, then for all $c\in[0,c_\mu)$,
\bess
\liminf_{t\to\yy}u(t,x)\ge U(x) {\rm ~ ~uniformly ~ in ~ }[0,ct].
\eess
Certainly, it suffices to prove that for any small $\ep>0$ satisfying $\sqrt{\ep}(1+2u^*)<1$, there exists a $T_3>T_2$ such that
\[u(t,x)\ge U(x)-\sqrt\ep ~ ~ {\rm for ~ }t\ge T_3, ~ x\in[0,ct].\]

Recall the lower solution constructed in Lemma \ref{l2.4}:
\[\underline{h}(t)=(1-\ep)c_{n,\mu}(t-T)+X_{\ep}+4n,\; ~ ~ {\rm and } ~ ~ \underline{u}(t,x)=(1-\ep)\phi_n(x-\underline{h}(t)),\]
where $n$, $X_{\ep}$ and $T$ are chosen as in the proof of Lemma \ref{l2.4}.
For any $c\in[0, c_\mu)$, we shrink $\ep$ such that $c<(1-2\ep)c_\mu$. Meanwhile, by Lemma \ref{l2.3}, we can enlarge $n$, say $n\ge N_1>N$, such that $c<(1-2\ep) c_{n,\mu}$ and $U_n(x)\ge U(x)-\ep$ for $x\ge0$ and $n\ge N_1$. As in the proof of Lemma \ref{l2.4}, we can find a $T>T_3$ such that
\bes\label{2.8}
u(t,x)\ge \underline{u}(t,x) ~ ~ {\rm for ~ }t\ge T, ~ x\in[X_{\ep},\underline{h}(t)].\ees
Recalling $c<(1-2\ep)c_{n,\mu}=\underline{h}'(t)$, we may let $t$ be large enough, say $t\ge T_4>T$, such that $ct<\underline{h}(t)$ for $t\ge T_4$. Then simple calculations yield
\bess
\max_{x\in[X_{\ep},\,ct]}|\underline{u}(t,x)-(1-\ep)U_n(x)|
&\le& (1-\ep)[(u^*_n-\phi_n(ct-\underline{h}(t)))+(u^*_n-U_n(X_{\ep}))]\\
&\le&(1-\ep)\kk(u^*_n-\phi_n(ct-\underline{h}(t))+\frac{\ep u^*_n}{2}\rr),
\eess
which, combined with $\phi_n(ct-\underline{h}(t))\to u^*_n$ as $t\to\yy$, leads to that there is some $T_5>T_4$ such that
\[\max_{x\in[X_{\ep},\,ct]}|\underline{u}(t,x)-(1-\ep)U_n(x)|\le(1-\ep)\ep u^*_n ~ ~ {\rm for ~ }t\ge T_5.\]
Together with \eqref{2.8} and the choice of $\ep$, we have that, for $t\ge T_5$ and $x\in[X_{\ep},\,ct]$,
\bes\label{2.9}
u(t,x)\ge(1-\ep)(U_n(x)-\ep u^*_n)\ge(1-\ep)(U(x)-\ep-\ep u^*_n)\ge U(x)-\sqrt{\ep}.
\ees
Note that $u(t,x)\to U(x)$ in $C_{\rm loc}([0,\yy))$ as $t\to\yy$. Thus we can find a $T_6>T_5$ such that $u(t,x)\ge U(x)-\sqrt{\ep}$ for $t\ge T_6$ and $x\in[0,X_{\ep}]$, which, combined with \eqref{2.9}, completes this step.

{\bf Step 3.} In this step, we prove that if {\bf (J1)} is violated by $J$, then for all $c\ge0$,
\bess
\liminf_{t\to\yy}u(t,x)\ge U(x) {\rm ~ ~uniformly ~ in ~ }[0,ct].
\eess
In fact, by the proof of \cite[Theorem 3.4]{LLW22} we know that if {\bf (J1)} does not hold, then $c_{n,\mu}\to\yy$ as $n\to\yy$. So for any $c\ge0$, we can take $n$ to be sufficiently large such that $c<(1-2\ep)c_{n,\mu}$. Then the desired result follows directly from similar arguments with those in step 2. The proof is complete.
\end{proof}

Now we discuss the asymptotical behaviors of the level set $E_{\lambda}(t)$ for $u$ of \eqref{1.1}.

\begin{lemma}\label{l2.6}Assume that spreading happens. Let $(u,h)$ be the unique solution of \eqref{1.1}. Then the following statements hold.
\begin{enumerate}[$(1)$]
\item For any $\lambda\in[U(0),u^*)$, we have
\[\lim_{t\to\yy}\inf E_{\lambda}(t)=X_{\lambda},\]
where $X_{\lambda}$ is uniquely determined by $U(X_{\lambda})=\lambda$. This clearly implies
\[\lim_{t\to\yy}\frac{\inf E_{\lambda}(t)}{t}=0.\]
\item For any $\lambda\in(0,U(0))$, we have
\bess
 \left\{\begin{array}{lll}
 \dd\lim_{t\to\yy}\frac{\inf E_{\lambda}(t)}{t}=c_\mu ~ ~ {\rm if ~ {\bf (J1)} ~ holds},\\[3mm]
  \dd\lim_{t\to\yy}\frac{\inf E_{\lambda}(t)}{t}=\yy ~ ~ {\rm if ~ {\bf (J1)} ~ is ~ violated},
  \end{array}\right.
 \eess
where ${\ccc c_\mu}$ is uniquely determined by the semi-wave problem \eqref{1.3}.
\item If {\bf (J1)} holds, then
\[\lim_{t\to\yy}\frac{\sup E_{\lambda}(t)}{t}=c_\mu,\;\;\forall\;\lambda\in (0,u^*).\]
\item If {\bf (J1)} does not hold, then
\[\lim_{t\to\yy}\frac{\sup E_{\lambda}(t)}{t}=\yy,\;\;\forall\;\lambda\in (0,u^*).\]
\end{enumerate}
\end{lemma}
\begin{proof}(1) Notice that $U(x)$ is strictly increasing to $u^*$ in $[0,\yy)$ and $U(0)>0$. For any given $\lambda\in[U(0),u^*)$, there exists a unique $X_{\lambda}\ge0$ such that $U(X_{\lambda})=\lambda$. For any small $\ep>0$, we know $u(t,X_{\lambda}+\ep)\to U(X_{\lambda}+\ep)>\lambda$ and $u(t,X_{\lambda}-\ep)\to U(X_{\lambda}-\ep)<\lambda$ as $t\to\yy$. Thus $\inf E_{\lambda}(t)<X_{\lambda}+\ep$ for all large $t$. On the other hand, since $u(t,x)\to U(x)$ uniformly in $[0,X_{\lambda}-\ep]$ for any small $\ep>0$ as $t\to\yy$, it follows that $\inf E_{\lambda}(t)>X_{\lambda}-\ep$ for all large $t$. Hence $X_{\lambda}-\ep<\inf E_{\lambda}(t)<X_{\lambda}+\ep$ for all large $t$, which obviously leads to the conclusion (1).

(2) We only handle the case where {\bf (J1)} holds since the other case can be proved similarly. For any given $c\in[0,c_\mu)$, it follows from Lemma \ref{l2.5} that $u(t,x)\to U(x)$ uniformly in $[0,ct]$ as $t\to\yy$, which implies $u(t,x)>\lambda\in(0,U(0))$ in $[0,ct]$ for all large $t$.  In light of $u(t,h(t))=0$ and the continuity of $u$ on $x$, we have $\inf E_{\lambda}(t)\in[ct,h(t)]$ for all large $t$. Owing to the arbitrariness of $c\in[0,c_{\mu})$ and ${h(t)}/{t}\to  c_\mu$ as $t\to\yy$, we immediately derive the desired result in conclusion (2).

(3) Similar to the arguments in the proof of (2), we have that for any $c\in[0, c_\mu)$, $u(t,x)\to U(x)$ uniformly in $[0,ct]$ as $t\to\yy$, which implies $u(t,ct)\to u^*$ as $t\to\yy$. Due to $u(t,h(t))=0$ and ${h(t)}/{t}\to c_\mu$ as $t\to\yy$, we complete the proof of (3).

(4) This result can be proved by using the analogous arguments as in the proof of (3). The details are omitted. \end{proof}

As an application of Theorem \ref{t1.1}, we next investigate the asymptotical behaviors for the solution of the following nonlocal diffusion problem on half line
 \bes\left\{\begin{aligned}\label{2.10}
&w_t= d\int_{0}^{\yy}J(x-y)w(t,y)\dy-dw+f(w), && t>0,\ x\ge0,\\
&w(0,x)=w_0(x),  && x\ge0,
 \end{aligned}\right.\ees
 where $w_0\in C_{c}([0,\yy))$ and $w_0(x)\ge0,\not\equiv0$ in $[0,\yy)$. As we see from \cite[Lemma 2.8]{LLW22}, the unique solution of \eqref{2.10} converges to $U(x)$ in $C_{\rm loc}([0,\yy))$ as $t\to\yy$. Now we prove the accurate asymptotical behaviors for \eqref{2.10}. To this aim, we recall some known results, which can be seen from \cite{DLZ,Ya}, about the travelling wave solution of the corresponding Cauchy of \eqref{2.10}.

 If $J$ satisfies the so-called ``thin tailed'' condition
 \begin{enumerate}
\item[({\bf J2})] there exists $\lambda>0$ such that
\[\displaystyle\int_{-\infty}^{\infty}\!\!J(x){\rm e}^{\lambda x}{\rm d}x<\infty,\]
 \end{enumerate}
then there exists a $c_*>0$ such that problem
\begin{eqnarray}\label{2.11}\left\{\begin{array}{lll}
 d\displaystyle\int_{-\infty}^{\infty}\!\!J(x-y)\varphi(y){\rm d}y-d\varphi+c\varphi'+f(\varphi)=0, \quad x\in\mathbb{R},\\[1mm]
\varphi(-\infty)=u^*,\ \ \varphi(\infty)=0
 \end{array}\right.
\end{eqnarray}
 has a non-increasing solution $\varphi_c$ with speed $c$ if and only if $c\ge c_*$. Moreover, $\varphi_c\in C^1(\mathbb{R})$ for any $c\ge c_*$. If $J$ does not satisfy ${\bf(J2)}$, then \eqref{2.11} does not have such non-increasing solution.

 Besides, in view of \cite[Theorem 5.1 and Theorem 5.2]{DLZ} , we have $c_\mu\nearrow c_*$ as $\mu\to\yy$ if ({\bf J2}) holds, while $c_\mu\nearrow\yy$ as $\mu\to\yy$ if ({\bf J2}) does not hold.
 Then we have the following result concerning the asymptotical behaviors of the solution of \eqref{2.10}. Denote the level set of $w$ by $\tilde{E}_{\lambda}(t)=\{x\ge0: w(t,x)=\lambda\}$ for any $\lambda\in(0,u^*)$.

 \begin{theorem}\label{t2.1}Let $w(t,x)$ be the unique solution of \eqref{2.10}. Then the following statements are valid.
 \begin{enumerate}[$(1)$]
 \item If {\bf (J2)} holds, then
  \bes\label{2.12}
 \left\{\begin{array}{lll}
\dd \lim_{t\to\yy}\max_{x\in[0,ct]}|w(t,x)-U(x)|=0, ~ ~\forall\, c\in[0,c_*),\\[1mm]
\dd \lim_{t\to\yy}\sup_{x\in[ct,\yy)}w(t,x)=0, ~ ~\forall\, c>c_*.
  \end{array}\right.
 \ees
 Moreover,
  \bes\label{2.13}
 \left\{\begin{array}{lll}
\dd \lim_{t\to\yy}\inf \tilde{E}_{\lambda}(t)=X_{\lambda}, ~ ~\forall\, \lambda\in[U(0),u^*),\\[3mm]
\dd \lim_{t\to\yy}\frac{\inf \tilde{E}_{\lambda}(t)}{t}=c_*, ~ ~\forall\, \lambda\in(0,U(0)),\\[3mm]
\dd \lim_{t\to\yy}\frac{\sup \tilde{E}_{\lambda}(t)}{t}=c_*, ~ ~\forall\, \lambda\in(0,u^*).
  \end{array}\right.
 \ees
 \item If {\bf (J2)} is violated, then
 \bess
 \dd \lim_{t\to\yy}\max_{x\in[0,ct]}|w(t,x)-U(x)|=0, ~ ~\forall\, c\in[0,\yy).
 \eess
 Besides,
 \bess
 \left\{\begin{array}{lll}
\dd \lim_{t\to\yy}\inf \tilde{E}_{\lambda}(t)=X_{\lambda}, ~ ~\forall\, \lambda\in[U(0),u^*),\\[2mm]
\dd \lim_{t\to\yy}\frac{\inf \tilde{E}_{\lambda}(t)}{t}=\yy, ~ ~\forall\, \lambda\in(0,U(0)),\\[3mm]
\dd \lim_{t\to\yy}\frac{\sup \tilde{E}_{\lambda}(t)}{t}=\yy, ~ ~\forall\, \lambda\in(0,u^*).
  \end{array}\right.
 \eess
 \end{enumerate}
 \end{theorem}

 \begin{proof}
 (1) Notice $w(t,x)\to U(x)$ in $C_{\rm loc}([0,\yy))$ as $t\to\yy$ and $U(x)$ is strictly increasing to $u^*$. Besides, it follows from a comparison argument that  $\limsup_{t\to\yy}w(t,x)\le u^*$ uniformly in $[0,\yy)$. Thus arguing as in the proof of Lemma \ref{l2.5}, we obtain
 \bes\label{2.14}
 \limsup_{t\to\yy}w(t,x)\le U(x) ~ ~ {\rm uniformly ~ in ~ }[0,\yy).
 \ees
By the maximum principle, $w(t,x)>0$ for $t>0$ and $x\ge0$. Then we can choose the initial function $u_0$ in \eqref{1.1} to be suitably small such that $u_0(x)\le w(1,x)$ for $x\in[0,h_0]$. The comparison principle yields $w(t+1,x)\ge u(t,x)$ for $t\ge0$ and $x\in[0,h(t)]$, where $(u,h)$ is the unique solution of \eqref{1.1} with $u_0$ chosen as above.

 For any $c\in[0,c_*)$, we can take $\mu$ to be sufficiently large such that spreading happens for \eqref{1.1} and $c_\mu>c$ since $c_\mu\nearrow c_*$ as $\mu\to\yy$ if {\bf (J2)} holds. Note that {\bf (J2)} implies {\bf (J1)}. By Lemma \ref{l2.5}, we know
 \[\dd\lim_{t\to\yy}\max_{x\in[0,ct]}|u(t,x)-U(x)|=0, ~ ~\forall\, c\in[0,c_\mu),\]
 which, combined with $w(t,x)\ge u(t,x)$ for $t\ge0$ and $x\in[0,h(t)]$, leads to
 \[\liminf_{t\to\yy}w(t,x)\ge U(x) ~ ~ {\rm uniformly ~ in }~[0,ct].\]
 Noticing  \eqref{2.14}, we immediately obtain the first limit of \eqref{2.12}.

Let $\hat{w}(t,x)=K\varphi_{c_*}(x-c_*t)$ where $K\gg1$ and $(c_*,\varphi_{c_*})$ is the solution pair of travelling wave problem \eqref{2.11}. Clearly, we can choose $K\gg1$ such that $\hat{w}(0,x)=K\varphi_{c_*}(x)\ge w_0(x)$ for $x\ge0$. Then it follows from a comparison method that $w(t,x)\le \hat{w}(t,x)$ for $t\ge0$ and $x\ge0$. Thanks $c>c_*$, the second limit of \eqref{2.12} is obtained.

Before proving the results about the asymptotical behaviors of the level set, we first show that for any $\lambda\in(0,u^*)$, $\tilde{E}_{\lambda}(t)$ is non-empty when $t$ is large enough. Consider the problem
  \bess\left\{\begin{aligned}
&\tilde w_t= d\int_{0}^{\yy}J(x-y)\tilde w(t,y)\dy-d\tilde w, && t>0,\ x\ge0,\\
&\tilde w(0,x)=w_0(x),  && x\ge0.
 \end{aligned}\right.\eess
In view of the Cauchy-Lipschitz-Picard Theorem (see \cite[Theorem 7.3]{Bre}, pp184), the above problem has a unique solution $\tilde w$ with $\tilde w(t,\cdot)\in L^1([0,\yy))$ for all $t\ge0$. Let $\bar{w}(t,x)={\rm e}^{f'(0)t}\tilde{w}$. Clearly, using a comparison argument one can deduce  $\bar{w}(t,x)\ge w(t,x)$ for $t\ge0$ and $x\ge0$, which implies $w(t,\cdot)\in L^1([0,\yy))$ for all $t\ge0$. Moreover, by the maximum principle, $w(t,x)>0$ for $t>0$ and $x\ge0$. Thus $\liminf_{x\to\yy}w(t,x)=0$ for all $t\ge0$. Together with the fact that $w(t,x)\to U(x)$ in $C_{\rm loc}([0,\yy))$ as $t\to\yy$ and $U(x)$ is strictly increasing to $u^*$, it follows that $\tilde{E}_{\lambda}(t)$ is non-empty when $t$ is large enough.

The first limit of \eqref{2.13} can be proved by using the similar arguments as in the proof of Lemma \ref{l2.6}. The last two limits of \eqref{2.13} directly follows from \eqref{2.12}. The conclusion (1) is proved.

(2) Note that $c_\mu\nearrow\yy$ as $\mu\to\yy$ if {\bf (J2)} does not hold. Then using the analogous methods as above, we can prove (2). The details are omitted.

Therefore, the proof is complete.
  \end{proof}

  Theorem \ref{t1.1} follows from Lemmas \ref{l2.4} and \ref{l2.5}; Theorem \ref{t1.2} follows from Lemma \ref{l2.6}.

\section{Accelerated spreading}
This section involves the rate of accelerated spreading when kernel $J$ satisfies ${\bf(J^\gamma)}$ with $\gamma\in(1,2]$. That is, we will prove Theorem \ref{t1.2}, which will be done by several lemmas.

 \begin{lemma}\label{l3.1}Let $(u,h)$ be the unique solution of \eqref{1.1} and ${\bf(J^\gamma)}$ hold with $\gamma\in(1,2]$. Then there exists $C>0$ such that when $t\gg1$,
  \bess\left\{\begin{aligned}
  &h(t)\le Ct^{\frac{1}{\gamma-1}} ~ ~ {\rm if ~ }\gamma\in(1,2),\\
  &h(t)\le Ct\ln t ~ ~ {\rm if ~ }\gamma=2.
  \end{aligned}\right.\eess
 \end{lemma}

 \begin{proof}If $\gamma\in(1,2)$, we have
 \bess
 \int_0^{h}\!\int_{h}^{\infty}\!J(x-y){\rm d}y{\rm d}x&=&\int_0^{1}J(y)y{\rm d}y+\int_{1}^{h}J(y)y{\rm d}y+h\int_{h}^{\infty}\!\!J(y){\rm d}y\approx h^{2-\gamma},
 \eess
and if $\gamma=2$,
\bess
\int_0^{h}\!\int_{h}^{\infty}\!J(x-y){\rm d}y{\rm d}x&=&\int_0^{1}J(y)y{\rm d}y+\int_{1}^{h}J(y)y{\rm d}y+h\int_{h}^{\infty}\!\!J(y){\rm d}y\approx\ln h.
 \eess
Moreover, in view of a comparison argument, there is a $T>0$ such that $u(t,x)\le2u^*$ for $t\ge T$ and $x\in[0,h(t)]$. Hence, for $t\ge T$,
\[h'(t)\le2\mu u^*\int_0^{h(t)}\!\!\!\int_{h(t)}^{\infty}\!J(x-y){\rm d}y{\rm d}x,\]
which immediately completes the proof.
 \end{proof}

The following lemma is crucial for the construction of the desired lower solution, and can be proved by the analogous methods with \cite[Lemma 5.4]{DN1}.

\begin{lemma}\label{l3.2}Assume that positive constants $k_1,k_2$ and $l$ satisfy $k_2>k_1>l$. Define
 \[\psi(x)=\min\kk\{1,\;\frac{k_2-x}{k_1}\rr\}.\]
Then for any given $\ep>0$, there exists a $k_0>l$, depending only on $l, \ep$ and $J$, such that when $k_1>k_0$ and $k_2-k_1>2k_0$, we have
 \[\int_{l}^{k_2}J(x-y)\psi(y)\dy\ge (1-\ep)\psi(x) ~ ~ {\rm for ~ }x\in[k_0,k_2].\]
 \end{lemma}

\begin{proof} By the condition {\bf (J)}, for any given $\ep>0$, there is a $l_0>l$ such that
 \[\int_{-l_0}^{l_0}J(y)\dy>1-\frac{\ep}{2}.\]
 Choose $k_0>2\max\{l_0,\frac{2l_0}{\ep}\}$. Let $k_1>k_0$ and $k_2-k_1>2k_0$. The following proof is divided into the four cases:
 \bess
&& {\bf Case ~  1:} ~ x\in[k_0,k_2-k_1-l_0]; ~ ~ ~ ~ ~ ~  ~ ~ {\bf Case ~  2:} ~ x\in[k_2-k_1-l_0,k_2-k_1+l_0];\\
 && {\bf Case ~  3:} ~ x\in[k_2-k_1+l_0,k_2-l_0]; ~ ~ {\bf Case  ~ 4:} ~ x\in[k_2-l_0,k_2].\eess

For Case 1, straightforward computations yield
 \bess
 \dd\int_{l}^{k_2}J(x-y)\psi(y)\dy&\ge&\int_{l}^{k_2-k_1}J(x-y)\dy
 =\int_{l-x}^{k_2-k_1-x}J(y)\dy\\
 &\ge&\int_{l-k_0}^{l_0}J(y)\dy\ge\int_{-l_0}^{l_0}J(y)\dy\\
 &\ge&1-\frac{\ep}{2}\ge(1-\ep)\psi(x).
 \eess

For Case 2, we have
 \bess
 \int_{l}^{k_2}J(x-y)\psi(y)\dy&=&\int_{l-x}^{k_2-x}J(y)\psi(x+y)\dy
 \ge\int_{l-k_2+k_1+l_0}^{k_1-l_0}J(y)\psi(x+y)\dy\\
 &\ge&\int_{-l_0}^{l_0}J(y)\psi(x+y)\dy
 \ge\int_{-l_0}^{l_0}J(y)\psi(k_2-k_1+l_0+y)\dy\\
 &=&\int_{-l_0}^{l_0}J(y)\dy-\int_{-l_0}^{l_0}J(y)\frac{l_0+y}{k_1}\dy\\
 &\ge&1-\frac{\ep}{2}-\frac{2l_0}{k_1}\ge 1-\ep\ge(1-\ep)\psi(x).
 \eess

For Case 3, simple calculations show
 \bess
 \dd\int_{l}^{k_2}J(x-y)\psi(y)\dy&=&\int_{l-x}^{k_2-x}J(y)\psi(x+y)\dy
 \ge\int_{l-k_2+k_1-l_0}^{l_0}J(y)\psi(x+y)\dy\\
 &\ge&\int_{-l_0}^{l_0}J(y)\psi(x+y)\dy
 =\int_{-l_0}^{l_0}J(y)\frac{k_2-x}{k_1}\dy-\int_{-l_0}^{l_0}J(y)\frac{y}{k_1}\dy\\
 &=&\int_{-l_0}^{l_0}J(y)\frac{k_2-x}{k_1}\dy
 =\psi(x)\int_{-l_0}^{l_0}J(y)\dy\\
 &\ge&(1-\ep)\psi(x).
 \eess

For Case 4, we see
 \bess
 \dd\int_{l}^{k_2}J(x-y)\psi(y)\dy&=&\int_{l-x}^{k_2-x}J(y)\psi(x+y)\dy
 \ge\int_{-l_0}^{k_2-x}J(y)\psi(x+y)\dy\\
 &=&\int_{-l_0}^{l_0}J(y)\psi(x+y)\dy-\int_{k_2-x}^{l_0}J(y)\psi(x+y)\dy\\
 &\ge&\int_{-l_0}^{l_0}J(y)\psi(x+y)\dy
 =\psi(x)\int_{-l_0}^{l_0}J(y)\dy\\
 &\ge&(1-\ep)\psi(x),
 \eess
where we have used $\psi(x+y)\le0$ when $y\in[k_2-x,l_0]$. Hence the proof is finished.
 \end{proof}

Now we are going to construct two suitable lower solutions to derive the results as wanted. This process is divided into two cases, case 1: $\gamma\in(1,2)$ and case 2: $\gamma=2$. Let us begin with handling case 1: $\gamma\in(1,2)$.

 \begin{lemma}{\label{l3.3}}Let $(u,h)$ be the unique solution of \eqref{1.1}. Suppose that ${\bf(J^\gamma)}$ holds with $\gamma\in(1,2)$. Then the following statements are valid.
 \begin{enumerate}[$(1)$]
 \item There exists a positive constant $C$ such that $h(t)\ge C t^{\frac{1}{\gamma-1}}$ for all large $t$.
 \item For any $s(t)=o(t^{\frac{1}{\gamma-1}})$, we have
 \[\liminf_{t\to\yy}u(t,x)\ge U(x) ~ ~ {\rm uniformly ~ in ~ }[0,s(t)].\]
 \end{enumerate}
 \end{lemma}
 \begin{proof}For any $\ep$ satisfying $0<\ep\ll1$, due to Lemma \ref{l2.2}, we can find a $X_{\ep}>0$ such that $U(x)\ge u^*-\sqrt{\ep}/2$ for $x\ge X_{\ep}$. Define
 \[\underline{h}(t)=(l_1t+\theta)^{\frac{1}{\gamma-1}} ~ ~ {\rm and ~ ~ }\underline{u}(t,x)=l_{\ep}\min\kk\{1,\;2\frac{\underline{h}(t)-x}
 {\underline{h}(t)}\rr\} ~ ~ {\rm for ~ }(t,x)\in[0,\yy)\times[X_{\ep},\underline{h}(t)],\]
where $l_{\ep}=u^*-\sqrt{\ep}$ and positive constants $l_1$ and $\theta$ are determined later.

 Taking advantage of Lemma \ref{l3.2} with $l=X_{\ep}$, we have that if $\theta^{\frac{1}{\gamma-1}}>4k_0$,
 \bes\label{3.1}
 \int_{X_{\ep}}^{\underline{h}(t)}J(x-y)\underline{u}(t,y)\dy\ge(1-\ep^2)\underline{u}(t,x) ~ ~ {\rm for ~ }t>0, ~ x\in[k_{0},\underline{h}(t)].
 \ees

 Next we prove that by choosing $\theta$ and $T$ to be large enough, as well as $l_1$ to be sufficiently small, there holds:
 \bes\label{3.2}
\left\{\begin{aligned}
&\underline u_t\le d\displaystyle\int_{X_{\ep}}^{\underline h(t)}J(x-y)\underline u(t,y){\rm d}y-d\underline u+ f(\underline u), && t>0,~x\in[k_0,\underline h(t))\setminus\left\{\frac{\underline{h}(t)}2\right\},\\
&\underline u(t,\underline h(t))=0,&& t>0,\\
&\underline h'(t)<\mu\displaystyle\int_{X_{\ep}}^{\underline h(t)}\int_{\underline h(t)}^{\infty}
J(x-y)\underline u(t,x){\rm d}y{\rm d}x,&& t>0,\\
&\underline{u}(t,x)\le u(t+T,x), && t>0, ~ x\in[X_{\ep},k_0],\\
&\underline h(0)<h(T),\;\;\underline u(0,x)\le u(T,x),&& x\in[X_{\ep},\underline h(0)].
\end{aligned}\right.
\ees
Once we have proved \eqref{3.2}, making use of Lemma \ref{l2.1} with $(s(t),r(t))=(X_{\ep},k_0)$, we derive
\[h(t+T)\ge \underline{h}(t) ~ ~ {\rm and ~ ~ }u(t+T,x)\ge \underline u(t,x) ~ ~ {\rm for ~ }t\ge0, ~ x\in[X_{\ep},\underline{h}(t)],\]
which obviously yields the assertion (1). As for the assertion (2), we have
\bess\max_{x\in[X_{\ep},s(t)]}|\underline{u}(t,x)-U(x)|\le (u^*-\sqrt{\ep})\kk(1-\min\kk\{1,\;
2\frac{\underline{h}(t)-s(t)}{\underline{h}(t)}\rr\}\rr)
+\frac{3\sqrt{\ep}}{2}\to\frac{3\sqrt{\ep}}{2}
\eess
as $t\to\yy$. Thus there exists a $T_1>T$ such that when $t\ge T_1$ and $x\in[X_{\ep},s(t)]$, $\underline{u}(t,x)\ge U(x)-2\sqrt{\ep}$. Recall that $u(t+T,x)\ge \underline u(t,x)$ for $t\ge0$ and $x\in[X_{\ep},\underline{h}(t)]$. We derive that $u(t,x)\ge U(x)-2\sqrt{\ep}$ for $t\ge T_1+T$ and $x\in[X_{\ep},s(t)]$. On the other hand, since $u(t,x)\to U(x)$ in $C_{\rm loc}([0,\yy))$ as $t\to\yy$, one can choose a $T_2>T_1+T$ such that $u(t,x)\ge U(x)-2\sqrt{\ep}$ for $t\ge T_2$ and $x\in[0,X_{\ep}]$. Hence for any small $\ep>0$, we deduce that $u(t,x)\ge U(x)-2\sqrt{\ep}$ for $t\ge T_2$ and $x\in[0,s(t)]$, which immediately indicates the assertion (2).

Now let us verify \eqref{3.2}. By the definition of $(\underline{u},\underline{h})$, the equality in the second line of \eqref{3.2} holds. Straightforward computations show that if $\theta$ is suitably large, then
\bess
\mu\displaystyle\int_{X_{\ep}}^{\underline h(t)}\!\int_{\underline h(t)}^{\infty}\!J(x-y)\underline u(t,x){\rm d}y{\rm d}x
&\ge&2\mu  l_{\ep}\displaystyle\int_{\frac{\underline{h}(t)}2}^{\underline h(t)}\!\int_{\underline h(t)}^{\infty}\!
J(x-y)\frac{\underline{h}(t)-x}{\underline{h}(t)}{\rm d}y{\rm d}x\\
&=&\frac{2\mu l_{\ep}}{\underline{h}(t)}
\dd\int_{-\frac{\underline{h}(t)}2}^{0}\int_0^{\infty}\!J(x-y)(-x){\rm d}y{\rm d}x\\
&=&\frac{2\mu  l_{\ep}}{\underline{h}(t)}\left(\int_0^{\frac{\underline{h}(t)}2}
\int_0^{y}\!\!J(y)x{\rm d}x{\rm d}y+\int_{\frac{\underline{h}(t)}2}^{\infty}
\int_0^{\frac{\underline{h}(t)}2}\!J(y)x{\rm d}x{\rm d}y\right)\\
 &\ge&\frac{\mu l_{\ep}}{\underline{h}(t)}\int_{\frac{\underline{h}(t)}
 {4}}^{\frac{\underline{h}(t)}2}J(y)y^2{\rm d}y\\
 &\ge&\frac{\varsigma_1\mu l_{\ep}}{\underline{h}(t)}\int_{\frac{\underline{h}(t)}{4}}^{\frac{\underline{h}(t)}2}
y^{2-\gamma}{\rm d}y\\
&\ge& C_1 l_{\ep}\mu \underline{h}^{2-\gamma}(t),
 \eess
where $C_1$ depends only on $J$. Besides, if $C_1 l_{\ep}\mu>\frac{ l_1}{\gamma-1}$, then
\[\underline{h}'(t)=\frac{ l_1}{\gamma-1}( l_1t+\theta)^{\frac{2-\gamma}{\gamma-1}}=\frac{ l_1}{\gamma-1}\underline{h}^{2-\gamma}(t)< C_1 l_{\ep}\mu \underline{h}^{2-\gamma}(t).\]
Thus the inequality in the third line of \eqref{3.2} is valid if $\theta$ is large enough and $l_1$ is sufficiently small.

Now we prove the inequality in the first line of \eqref{3.2}. For $x\in[{\underline{h}(t)}/4,\underline{h}(t)]$, if $\theta$ is large enough, we see
  \bes\label{3.3}
\int_{X_{\ep}}^{\underline h(t)}J(x-y)\underline u(t,y){\rm d}y&=&\int_{X_{\ep}-x}^{\underline h(t)-x}J(y)\underline u(t,x+y){\rm d}y\nonumber\\
  &\ge&\int_{X_{\ep}-\frac{\underline{h}(t)}{4}}^{-\frac{\underline{h}(t)}{8}}J(y)\underline u(t,x+y){\rm d}y\nonumber\\
  &\ge& l_{\ep}\int_{-\frac{\underline{h}(t)}{6}}^{-\frac{\underline{h}(t)}{8}}
 \frac{\varsigma_1}{|y|^{\gamma}}\frac{\underline{h}(t)-x-y}{\underline{h}(t)}{\rm d}y\nonumber\\
&\ge&\frac{l_{\ep}}{\underline{h}(t)}\int_{-\frac{\underline{h}(t)}{6}}^{-\frac{\underline{h}(t)}{8}}
\frac{\varsigma_1}{|y|^{\gamma}}(-y){\rm d}y\nonumber\\
&\ge& C_2 l_{\ep}\underline{h}^{1-\gamma}(t)
\ees
with $C_2$ relying only on $J$. According to {\bf (F)}, it is easy to see that there exists a $\rho>0$, depending only on $f$, such that
\bes\label{3.4}
f(x)\ge \rho\min\{x,\,u^*-x\} ~ ~ {\rm for ~ }x\in[0,u^*].
\ees
Thus we have
\bes\label{3.5}
\left\{\begin{aligned}
&f(\underline u)\ge \rho\min\{\underline u,\,u^*-\underline u\}\ge \rho\min\{ l_{\ep},\,u^*-\underline u\}\ge \rho\sqrt{\ep}\; ~ ~ {\rm for ~ }x\in\kk[k_0,\,\frac{\underline{h}(t)}{4}\rr],\\
&f(\underline u)\ge \rho\min\{\underline u,\,u^*-\underline u\}\ge \rho\ep\underline{u} \;~ ~ {\rm if } ~ \ep ~ {\rm is} ~ {\rm small} ~ {\rm enough}.
\end{aligned}\right.
 \ees
Using \eqref{3.1}, \eqref{3.3} and \eqref{3.5} yields that for $x\in[k_0,\frac{\underline{h}(t)}{4}]$,
\bess
d\displaystyle\int_{X_{\ep}}^{\underline h(t)}J(x-y)\underline u(t,y){\rm d}y-d\underline u+ f(\underline u)\ge-d\ep^2\underline{u}+\rho\sqrt{\ep}\ge-d\ep^2 u^*+\rho\sqrt{\ep}\ge0,
\eess
and for $x\in [{\underline{h}(t)}/{4},\underline{h}(t)]$,
\bess
 &&d\displaystyle\int_{X_{\ep}}^{\underline h(t)}J(x-y)\underline u(t,y){\rm d}y-d\underline u+ f(\underline u)\nonumber\\
&\ge&\left(\min\left\{\frac{\rho\ep}2,d\right\}+\left(d-\frac{\rho\ep}2\right)^+\right)
\int_{X_{\ep}}^{\underline h(t)}J(x-y)\underline u(t,y){\rm d}y-\left(d-\rho\ep\right)\underline u\nonumber\\
&\ge&\min\left\{\frac{\rho\ep}2,\,d\right\}C_2 l_{\ep}\underline{h}^{1-\gamma}(t)+\left(d-\frac{\rho\ep}2\right)^+(1-\ep^2)\underline{u}
-\left(d-\rho\ep\right)\underline u\nonumber\\[1mm]
&\ge&\min\left\{\frac{\rho\ep}2,\,d\right\}C_2  l_{\ep}\underline{h}^{1-\gamma}(t)
\eess
provided that $\ep$ is suitably small. On the other hand, we have $\underline{u}_t(t,x)=0$ for $t>0$ and $x\in[0,{\underline{h}(t)}/2]$, and
\[\underline{u}_t(t,x)=2 l_{\ep}\frac{x\underline{h}'(t)}{\underline{h}^2(t)}\le 2 l_{\ep}\frac{\underline{h}'(t)}{\underline{h}(t)}=\frac{2 l_1 l_{\ep}}{\gamma-1}\underline{h}^{1-\gamma}(t) ~ ~ {\rm for} ~ t>0, ~ x\in\kk(\frac{\underline{h}(t)}2,\underline{h}(t)\rr).\]
So the inequality in the first line of \eqref{3.2} holds if $\frac{2 l_1}{\gamma-1}\le\min\left\{\frac{\rho\ep}2,d\right\}C_2$.

Now to prove \eqref{3.2}, it remains to show the inequalities in the last two lines of \eqref{3.2}. For positive constants $\ep,\;\theta$ and $l_1$ chosen as above, since spreading happens for \eqref{1.1}, we can find a $T>0$ such that $h(T)>\theta^{\frac{1}{\gamma-1}}=\underline{h}(0)$ and $u(t,x)\ge U(x)-\sqrt{\ep}/2$ for $t\ge T$ and $x\in[X_{\ep}, \underline{h}(0)]$. Recalling $U(x)\ge u^*-\sqrt{\ep}/2$ for $x\ge X_{\ep}$, we have that when $t\ge T$ and $x\in[X_{\ep}, \underline{h}(0)]$,
\[u(t,x)\ge U(x)-\sqrt{\ep}/2\ge u^*-\sqrt{\ep}\ge \underline{u}(t,x),\]
which immediately implies that $u(t+T,x)\ge\underline{u}(t,x)$ for $t\ge0$ and $x\in[X_{\ep},\underline{h}(0)]$. Thus \eqref{3.2} holds. The proof is complete.
 \end{proof}

Now we deal with case 2: $\gamma=2$.

  \begin{lemma}{\label{l3.4}}Let $(u,h)$ be the unique solution of \eqref{1.1}. Suppose that ${\bf(J^\gamma)}$ holds with $\gamma=2$. Then the following statements are valid.
 \begin{enumerate}[$(1)$]
 \item There exists a positive constant $C$ such that $h(t)\ge C t\ln t$ for all large $t$.
 \item For any $s(t)=o(t\ln t)$, we have
 \[\liminf_{t\to\yy}u(t,x)\ge U(x) ~ ~ {\rm uniformly ~ in ~ }[0,s(t)].\]
 \end{enumerate}
 \end{lemma}

\begin{proof} Let $X_{\ep}$ and $k_0$ be defined as in the proof of Lemma \ref{l3.3} for any small $\ep>0$. Recall that $U(x)\ge u^*-\sqrt{\ep}/2$ for $x\ge X_{\ep}$. Define
\[\underline{h}(t)=l_1(t+\theta)\ln (t+\theta) ~ ~ {\rm and ~ ~ }\underline{u}(t,x)=l_{\ep}\min\kk\{1,\,\frac{\underline{h}(t)-x}
{(t+\theta)^{\alpha}}\rr\} ~ ~ {\rm for ~ }t\ge0, ~ x\in[X_{\ep},\underline{h}(t)],\]
where $l_{\ep}=u^*-\sqrt{\ep}$ and $\alpha\in(0,1)$. Moreover, $l_1$ and $\theta$ satisfy
\bess
&\theta\gg1, ~ \;\theta^{\alpha}>k_0, ~\; l_1\theta^{1-\alpha}>3\alpha, ~\; l_1\theta\ln \theta-3\theta^{\alpha}>2k_0,\\[1mm]
&\dd\frac{ l_{\ep}\mu \varsigma_1}2(1-\alpha)>2 l_1, \;\;~ \min\kk\{\frac{\rho\ep}{2},\,d\rr\}\frac{l_{\ep}\zeta_1\alpha}{4}\ge l_1(2l_{\ep}-\alpha l_1+1),
\eess
where $\zeta_1$ is given by ${\bf (J^{\gamma})}$ and $\rho$ is determined by \eqref{3.4}.
In view of the above choices of $l_1$ and $\theta$ and Lemma \ref{l3.2}, we have
\bes\label{3.6}
\int_{X_{\ep}}^{\underline{h}(t)}J(x-y)\underline{u}(t,y)\dy\ge(1-\ep^2)\underline{u}(t,x) ~ ~ {\rm for ~ }t\ge0,~ x\in[k_0,\underline{h}(t)].
\ees
Now we are in the position to show that by further enlarging $\theta$ suitably and choosing $T$ large enough, there holds:
\bes\label{3.7}
\left\{\begin{aligned}
&\underline u_t\le d\displaystyle\int_{X_{\ep}}^{\underline h(t)}J(x-y)\underline u(t,y){\rm d}y-d\underline u+ f(\underline u), && t>0,~x\in[k_0,\underline h(t))\setminus\left\{\underline{h}(t)-(t+\theta)^{\alpha}\right\},\\
&\underline u(t,\underline h(t))=0,&& t>0,\\
&\underline h'(t)<\mu\displaystyle\int_{X_{\ep}}^{\underline h(t)}\int_{\underline h(t)}^{\infty}
J(x-y)\underline u(t,x){\rm d}y{\rm d}x,&& t>0,\\
&\underline{u}(t,x)\le u(t+T,x), && t>0, ~ x\in[X_{\ep},k_0],\\
&\underline h(0)<h(T),\;\;\underline u(0,x)\le u(T,x),&& x\in[X_{\ep},\underline h(0)].
\end{aligned}\right.
\ees
If we have done it, by the definitions of $(\underline{u},\underline{h})$ and the similar arguments as in the proof of Lemma \ref{l3.3}, we readily derive the results as wanted.

Now let us verify \eqref{3.7}. The second identity of \eqref{3.7} is obvious. Direct calculations show
\bess
\mu\displaystyle\int_{X_{\ep}}^{\underline h(t)}\!\!\int_{\underline h(t)}^{\infty}\!
J(x-y)\underline u(t,x){\rm d}y{\rm d}x
&\ge&\mu l_{\ep}\displaystyle\int_{\frac{\underline{h}(t)}2}^{\underline{h}(t)-(t+\theta)^{\alpha}}\!\int_{\underline h(t)}^{\infty}\!J(x-y){\rm d}y{\rm d}x\\
&=& l_{\ep}\mu\left(\int_{(t+\theta)^{\alpha}}^{\frac{\underline{h}(t)}2}
\int_{(t+\theta)^{\alpha}}^{y}
J(y){\rm d}x{\rm d}y+\int_{\frac{\underline{h}(t)}2}^{\infty}
\int_{(t+\theta)^{\alpha}}^{\frac{\underline{h}(t)}2}
J(y){\rm d}x{\rm d}y\right)\\
 &\ge& l_{\ep}\mu\int_{(t+\theta)^{\alpha}}^{\frac{\underline{h}(t)}2}\int_{(t+\theta)^{\alpha}}^{y}
J(y){\rm d}x{\rm d}y\\
 &\ge& l_{\ep}\mu\int_{2(t+\theta)^{\alpha}}^{\frac{\underline{h}(t)}2}
J(y)\left[y-(t+\theta)^{\alpha}\right]{\rm d}y\\
 &\ge& \frac{ l_{\ep}\mu \varsigma_1}2\int_{2(t+\theta)^{\alpha}}^{\frac{\underline{h}(t)}2}
y^{-1}{\rm d}y\\
&=&\frac{ l_{\ep}\mu \varsigma_1}2[\ln l_1+(1-\alpha)\ln(t+\theta)+\ln\ln(t+\theta)-\ln2]\\
&\ge&\frac{ l_{\ep}\mu \varsigma_1}2(1-\alpha)\ln(t+\theta)
\eess
provided that $\theta$ is large enough such that
\[J(y)\ge\zeta_1y^{-2}\;\; {\rm ~ for ~ }y\in\kk[\frac{l_1\theta\ln \theta}{2},\;2\theta^{\alpha}\rr],~ ~ {\rm and ~ ~ }\ln\ln\theta>\ln_2-\ln l_1.\]
 Moreover, it is easy to see that
\[\underline{h}'(t)= l_1\ln(t+\theta)+ l_1\le 2 l_1\ln(t+\theta)<\frac{ l_{\ep}\mu \varsigma_1}2(1-\alpha)\ln(t+\theta).\]
Hence the inequality in the third line of \eqref{3.7} holds.

Now we prove the first inequality of \eqref{3.7}. Clearly,
 \bess
 \underline{u}(t,x)\ge l_{\ep}\frac{\underline{h}-x}{2(t+\theta)^{\alpha}}\;\;\;{\rm for} \;\;x\in[\underline{h}(t)-2(t+\theta)^{\alpha},\underline{h}(t)].
  \eess
Thus, for $x\in[\underline{h}(t)-(t+\theta)^{\alpha},\underline{h}(t)]$, we have
 \bess
\int_{X_{\ep}}^{\underline h(t)}J(x-y)\underline u(t,y){\rm d}y&=&\int_{X_{\ep}-x}^{\underline h(t)-x}J(y)\underline u(t,x+y){\rm d}y\\
 &\ge&\frac{ l_{\ep}}2\int_{-(t+\theta)^{\alpha}}^{-(t+\theta)^{\alpha/2}}J(y)
 \frac{\underline{h}(t)-x-y}{(t+\theta)^{\alpha}}{\rm d}y\\
&\ge&\frac{ l_{\ep}}2\int_{-(t+\theta)^{\alpha}}^{-(t+\theta)^{\alpha/2}}J(y)
\frac{-y}{(t+\theta)^{\alpha}}{\rm d}y\\
 &\ge&\frac{ l_{\ep}\varsigma_1}2\int_{-(t+\theta)^{\alpha}}^{-(t+\theta)^{\alpha/2}}
 \frac{(-y)^{-1}}{(t+\theta)^{\alpha}}{\rm d}y\\
&\ge&\frac{ l_{\ep}\varsigma_1\alpha\ln(t+\theta)}{4(t+\theta)^{\alpha}},
\eess
where $\theta$ is sufficiently large such that $J(y)\ge \zeta_1 y^{-2}$ for $y\le -\theta^{\frac{\alpha}{2}}$.
Moreover, for $x\in\big[\frac{\underline{h}(t)-(t+\theta)^{\alpha}}2,\, \underline{h}(t)-(t+\theta)^{\alpha}\big]$ with $\theta$ chosen as above, we have
\bess
 \int_{X_{\ep}}^{\underline h(t)}J(x-y)\underline u(t,y){\rm d}y&=&\int_{X_{\ep}-x}^{\underline h(t)-x}J(y)\underline u(t,x+y){\rm d}y\\
  &\ge&  l_{\ep}\int_{-(t+\theta)^{\alpha}}^{-(t+\theta)^{\alpha/2}}J(y){\rm d}y\\
&\ge & l_{\ep}\varsigma_1\int_{-(t+\theta)^{\alpha}}^{-(t+\theta)^{\alpha/2}}(-y)^{-2}{\rm d}y\\
 &\ge& \frac{ l_{\ep}\varsigma_1\alpha\ln(t+\theta)}{2(t+\theta)^{\alpha}}.
\eess
For $x\in\big[k_0,\,\frac{\underline{h}(t)-(t+\theta)^{\alpha}}2\big]$, make use of the first inequality of \eqref{3.5} and \eqref{3.6} it yields
\[d\displaystyle\int_{X_{\ep}}^{\underline h(t)}J(x-y)\underline u(t,y){\rm d}y-d\underline u+ f(\underline u)\ge-\ep^2 du^*+\rho\sqrt{\ep}\ge0 ~ ~ {\rm if } ~ \ep ~ {\rm is} ~ {\rm small} ~ {\rm enough}.\]
For $x\in\big[\frac{\underline{h}(t)-(t+\theta)^{\alpha}}2,\,\underline{h}(t)\big]$, taking advantage of the second inequality of \eqref{3.5} and \eqref{3.6}, and using the analogous arguments as in the proof of Lemma \ref{l3.3}, we can deduce
\bess
 d\displaystyle\int_{X_{\ep}}^{\underline h(t)}J(x-y)\underline u(t,y){\rm d}y-d\underline u+ f(\underline u)\ge\min\left\{\frac{\rho\ep}2,d\right\}\frac{ l_{\ep}\varsigma_1\alpha\ln(t+\theta)}{4(t+\theta)^{\alpha}}\;\; ~ {\rm if } ~ \ep\ll1.
\eess
On the other hand, we have $\underline{u}_t=0$ for $t>0$ and $x\in[0,\underline{h}(t)-(t+\theta)^{\alpha})$, and
 \[\underline{u}_t=\frac{ l_1 l_{\ep}(1-\alpha)\ln(t+\theta)+ l_1 l_{\ep}}{(t+\theta)^{\alpha}}
 +\frac{ l_{\ep}\alpha x}{(t+\theta)^{1+\alpha}}\le\frac{\ln(t+\theta)}{(t+\theta)^{\alpha}}\left[ l_1 l_{\ep}(1-\alpha)
 + l_1 l_{\ep}+ l_1\right]\]
for $t>0$ and $x\in(\underline{h}(t)-(t+\theta)^{\alpha},\, \underline{h}(t)]$. So the first inequality of \eqref{3.7} holds.

For $l_1$ and $\theta$ chosen as above, since spreading happens for \eqref{1.1}, there exists a $T>0$ such that $h(T)>l_1\theta\ln \theta=\underline{h}(0)$ and $u(t,x)\ge U(x)-\sqrt{\ep}/2$ for $t\ge T$ and $x\in[X_{\ep}, \underline{h}(0)]$. Notice that $U(x)\ge u^*-\sqrt{\ep}/2$ for $x\ge X_{\ep}$. We obtain that, for $t\ge T$ and $x\in[X_{\ep}, \underline{h}(0)]$,
\[u(t,x)\ge U(x)-\sqrt{\ep}/2\ge u^*-\sqrt{\ep}\ge \underline{u}(t,x),\]
which immediately gives $u(t+T,x)\ge\underline{u}(t,x)$ for $t\ge0$ and $x\in[X_{\ep},\underline{h}(0)]$. Hence \eqref{3.7} holds. The proof is ended.
\end{proof}

Clearly, Theorem \ref{t1.3} follows from Lemmas \ref{l3.1}, \ref{l3.3}, \ref{l3.4} and the fact that $\limsup_{t\to\yy}u(t,x)\le U(x)$ uniformly in $[0,\yy)$ which is already proved in the step 1 of the proof of Lemma \ref{l2.5}.

\end{document}